\documentclass[11pt,regno]{amsart}
\usepackage{euscript,graphicx,amscd,amsgen,amsfonts,amssymb,latexsym,
amsmath,amsthm,graphicx,mathrsfs,overpic}

\newtheorem{theorem}{Theorem}[section]
\newtheorem{lemma}{Lemma}
\newtheorem{corollary}{Corollary}

\theoremstyle{definition}

\newtheorem{example}{Example}
\newtheorem{remark}{Remark}

\newcommand{\eqdef}{\stackrel{\scriptscriptstyle\rm def}{=}}

\DeclareMathOperator{\Jac}{Jac}

\DeclareMathOperator{\llangle}{\langle\hspace{-0.05cm}\langle}
\DeclareMathOperator{\rrangle}{\rangle\hspace{-0.05cm}\rangle}

\def\bN{\mathbb{N}}

\def\bR{\mathbb{R}}

\def\cK{\EuScript{K}}

\def\cP{\EuScript{P}}
\def\ccP{\mathcal{P}}
\def\cR{\EuScript{R}}
\def\cH{\EuScript{H}}
\def\cU{\EuScript{U}}

\def\cM{\mathcal{M}}

\author[K. Gelfert]{Katrin Gelfert}
\address{Instituto de Matem\'atica, Universidade Federal do Rio de Janeiro, Cidade Universit\'aria - Ilha do Fund\~ao, Rio de Janeiro 21945-909,  Brazil}
\email{gelfert@im.ufrj.br}
\author[B. Schapira]{Barbara Schapira} 
\address{LAMFA UMR 7352 CNRS, Facult\'e de math\'ematiques et informatique,  Universit\'e de Picardie Jules Verne, 33 rue Saint Leu, 80 000 Amiens, France}
\email{barbara.schapira@u-picardie.fr}

\thanks{BS thanks the ANR grant ANR JCJC-0108 GEODE for its support during the realization of this work.
 KG was partially supported by CNPq and DynEurBraz.}

\title[Pressures for rank one geodesic flows]{Pressures for geodesic flows\\ of rank one manifolds}


\begin{document}

\maketitle 
 
\begin{abstract}  
We study the geodesic flow on the unit tangent bundle of a rank one manifold and we give 
conditions under which all classical definitions of pressure of a H\"older continuous potential coincide. 
We provide a large deviation statement, which allows to neglect (periodic) orbits that lack sufficient hyperbolic behavior. 
Our results involve conditions on the potential, that take into consideration 
its properties in the nonhyperbolic part of the manifold. We draw some conclusions for the construction of equilibrium states.
\end{abstract}


\section{Introduction}

We are interested in the thermodynamical formalism for  the geodesic flow $G=(g^t)_{t\in\bR}$ 
on the unit tangent bundle $T^1M$ of a smooth compact nonpositively curved manifold $M$. 
More precisely, we shall investigate different notions of pressure, study equilibrium states, 
and relate them to large deviation results. 
 
 Given a continuous map $\varphi\colon T^1M\to\bR $ (also called \emph{potential}),  
the {\em variational pressure} of $\varphi$ (with respect to the flow) is defined as
\begin{equation}\label{def-variational-pressure}
	P_\cM(\varphi)
	\eqdef\sup_{\mu\in\cM} \left(h(\mu)+\int_{T^1M}\varphi\,d\mu\right),
\end{equation}
where $\cM$ is the set of invariant probability measures under the geodesic flow 
$(g^t)_{t\in\bR}$ and $h(\mu)$ is the entropy of the time-1 map $g^1$ 
with respect to the flow-invariant probability measure $\mu$. 
An   {\em equilibrium state} for $\varphi$ is  an invariant probability measure 
$m_\varphi$ realizing the maximum of the pressure $P_\cM(\varphi)$.  
A \emph{measure of maximal entropy} is an invariant probability measure 
realizing the maximum of the \emph{topological entropy} $h=h_{top}(T^1M)=P_\cM(0)$.

When the manifold $M$ has negative curvature, the geodesic flow on $T^1M$ is Anosov. 
In this case, it is well known that any H\"older potential admits 
a unique equilibrium state and that this equilibrium state has very strong ergodic properties: 
it is ergodic, mixing, and possesses a local product structure~\cite{BowRue:75}. 
The measure of maximal entropy is the equilibrium measure of any constant potential, 
while the Liouville measure is the unique equilibrium state of 
a certain H\"older continuous potential $\varphi^{(u)}$ 
that is defined by means of the derivative of the Jacobian of the geodesic flow 
restricted to the unstable foliation (see~\eqref{phidef} for a precise definition). 

In this paper we consider compact {\em rank one} manifolds, that is, smooth compact
 nonpositively curved manifolds where   the curvature can vanish,
but where there exists at least one {\em rank one vector}, that is, a vector whose geodesic orbit 
does not bound any (even infinitesimal) flat strip (see Section \ref{geometric-preliminaries} for details). 
The \emph{regular set}  $\cR$  is the set of all rank one vectors and  the  \emph{higher rank set}
    $\cH$ is its  complementary set.
In this situation, the above potentials and their equilibrium states behave quite differently. 
 The potential $\varphi^{(u)}$ is still continuous, but probably not H\"older continuous in general. 
Besides the Liouville measure, 
 it admits trivial equilibrium measures, supported by the periodic orbits staying 
in the flat part $\cH$  of the manifold.  
Moreover, the ergodicity of the Liouville measure remains up to now a difficult open question. 
The measure of maximal entropy is known to exist (as does, in fact,
the equilibrium measure of any continuous potential) and to be unique~\cite{Kni:98}.
 Beyond these  few examples, little is known about 
uniqueness and further properties of equilibrium states.  

Motivated by the study of equilibrium measures for  (H\"older) continuous potentials, we 
need a complete understanding of the pressure function.
Therefore, we investigate in greater detail different notions of pressure, and
their relations to each other. 
The {\em pressure} of a potential measures the exponential growth rate of the complexity 
of the dynamics, weighted by this potential. There are several notions of  
 pressure,  that are known to coincide for sufficiently hyperbolic systems
(see \cite{PauPolSch:13} in the case of the geodesic flow in negative curvature).    
In Theorem~\ref{equality-pressure-surfaces}, we clarify their relations  
and prove that for certain continuous potentials, 
all existing definitions of pressure coincide also for geodesic flows of rank one surfaces. 

Let us introduce briefly these pressures (see Section \ref{pressure} for details). 
The {\em variational pressure} $P_{\cM}(\varphi)$ has been defined in~\eqref{def-variational-pressure}. 
The {\em topological pressure} $P_{top}(\varphi)$ is defined as 
the exponential growth rate of the values of $\varphi$ along orbits 
that are separated through the dynamics. 
The {\em Gurevich pressure} (or \emph{periodic orbits pressure}) $P_{Gur}(\varphi)$ 
is the exponential growth rate of the values of $\varphi$ along periodic orbits of increasing period. 
The regular Gurevich pressure $P_{Gur, \cR}(\varphi)$ is the exponential growth rate of the values of $\varphi$ along \emph{regular}
periodic orbits of increasing period.  
Finally,  the \emph{critical exponent} $\delta_{\Gamma,\varphi}$  is defined by means of  the fundamental group $\Gamma$
of $M$ acting by isometries on the universal cover  $\widetilde{M}$ of $M$.
Let $\widetilde{\varphi}$ be the $\Gamma$-invariant lift of $\varphi$ to $T^1\widetilde{M}$. 
We denote by $\delta_{\Gamma,\varphi,x}$  the {\em critical exponent}  of the
 Poincar\'e series $\sum_{\gamma\in\Gamma}e^{\int_x^{\gamma x}\widetilde{\varphi}}$, 
where the integral of  $\widetilde{\varphi}$ is taken along the geodesic path joining $x$ to $\gamma x$. 
Denote by $\delta_{\Gamma,\varphi}$ the supremum of these critical exponents over all $x\in\widetilde{M}$.   

In \cite{BurGel:14}, Burns and Gelfert show that if $M$ is a compact rank one surface, 
then there exists an increasing sequence $(\Lambda_k)_{k\in\bN}$ of basic sets 
whose union is dense in $T^1M$ and contains all regular periodic orbits and that 
topological pressure of appropriate $\varphi$ on $\Lambda_k$ converges to the pressure of $\varphi$ on $T^1M$. 
We  complete this result and study how closed geodesics contribute in this picture.
In restriction to  a basic set $\Lambda\subset T^1M$, 
for a  H\"older potential  $\varphi$ all above introduced pressures of the flow restricted 
to $\Lambda$ and, in particular Gurevich pressure,  
coincide and we will denote them shortly by $P(\varphi,\Lambda)$.  

For $T>0$, denote by $\Pi_{\cR}(T-1,T)$ (resp. $\Pi_{\cR}(T)$)  the set of primitive rank one 
  periodic orbits $\beta$ in $\cR$   of period  $\ell(\beta)\in[T-1,T)$ (resp. $\ell(\beta)< T$). 
In the case of higher rank periodic orbits, there can exist infinitely many parallel periodic orbits of same length, and 
we denote by $\Pi_{\cH}(T-1,T)$  (resp. $\Pi_{\cH}(T)$) a set of representatives of each homotopy class of primitive 
periodic orbits of $\cH$ of period $\ell(\beta)\in[T-1,T)$
 (resp. $\ell(\beta)< T$). 
Denote by $\#$ the cardinality.

There are examples due to Gromov (see Section~\ref{examples}) where the set $\cH$ carries positive entropy.
 However, it is always strictly less than the full topological entropy. 
More precisely, the following quantity is well-defined and positive (see Theorem~\ref{Knieper-thm} for details). 
\begin{equation}\label{defesspilon}
		\varepsilon_0\eqdef
		\limsup_{T\to +\infty}\frac{1}{T}\log 
			\frac{\#\Pi_\cR(T-1,T)}{\#\Pi_{\cH}(T-1,T)}
		=\limsup_{T\to +\infty}\frac{1}{T}\log 
			\frac{\#\Pi_\cR(T)}{\#\Pi_{\cH}(T)}	\,.
\end{equation}

The pressure of a continuous potential $\varphi\colon T^1M\to\bR$ measures the complexity of the dynamics weighted by the 
potential $\varphi$, whereas this complexity, topologically, comes mainly from the hyperbolic part $\cR$ of the manifold. 
Therefore, the different notions of pressures will behave reasonably, as soon as  $\varphi$ is not too large on $\cH$, 
in comparison with the quantity $\varepsilon_0$ defined above. This is   condition~\eqref{con-pot} for $\varphi$ below, where $\cM(\cH)$ (resp. $\cM(\cR)$)
denotes the set of invariant probability measures giving full measure to $\cH$ (resp. $\cR$).

\begin{theorem}\label{equality-pressure-surfaces} 
	Let $M$ be a smooth compact rank one manifold and let 
$\varphi\colon T^1M\to \bR$ be a continuous map.  Then the following inequalities hold 
	\begin{equation}\label{inequality-pressure-manifolds}
	P_{Gur, \cR}(\varphi) 
	\le P_{Gur}(\varphi)
	\le   \delta_{\Gamma,\varphi }
	\le P_{top}(\varphi)
	=P_{\cM}(\varphi)\,.
	\end{equation}
Moreover, if the potential $\varphi$ satisfies 
	\begin{equation}\label{con-pot}
		\max_{\mu\in\cM(\cH)}\int\varphi\,d\mu
		<\inf_{\mu\in\cM(\cR)}\int\varphi\,d\mu+\varepsilon_0\,,
	\end{equation}
then  we have
	\begin{equation}\label{equality-Gurevic}
	P_{Gur, \cR}(\varphi) 
	=P_{Gur}(\varphi)
	\end{equation}

If $M$ is a surface and $\varphi$ is H\"older continuous and such that 
$\varphi|_\cH$ is constant  then all notions of pressure coincide: 
	\begin{equation}\label{eqtheoremsurface}
	P_{Gur, \cR}(\varphi)
	= P_{Gur}(\varphi)
	= \delta_{\Gamma,\varphi }
	= P_{top}(\varphi)
	= \sup_{k \ge1}P_{top}(\varphi, \Lambda_k)
	= P_{\cM}(\varphi)\,.
\end{equation}  
\end{theorem}
Moreover, even if (\ref{con-pot}) is probably not optimal, it is certainly relevant. Indeed, we provide in Section~\ref{examples}
an example where (\ref{con-pot}) is not satisfied, (\ref{equality-Gurevic}) holds, but (\ref{eqtheoremsurface}) does not hold, 
because of a strict inequality $P_{Gur}(\varphi)<P_{\cM}(\varphi)$. 

Observe that \eqref{con-pot} holds for any non-negative potential $\varphi$
 which vanishes on the higher rank set $\cH$. In particular, 
  the potential $\varphi^{(u)}$ as well as   any constant potential satisfy \eqref{con-pot}.
In the case of the zero potential,~\eqref{equality-Gurevic} is due to Knieper~\cite{Kni:98} (see Theorem \ref{Knieper-thm}).  
Equality \eqref{equality-Gurevic} says that   good potentials do not see singular periodic orbits, ``good''  meaning  
in particular those that are positive and smaller on $\cH$  than the difference $\varepsilon_0$ of growth rates 
of regular and singular periodic orbits.

Our second result shows another way to avoid zero curvature, by considering only periodic orbits that are sufficiently hyperbolic 
from the point of view of their Lyapounov exponent.  
To obtain this, we apply {\em large deviation techniques}. 
 Denote by $\chi(\beta)$ the smallest positive
 Lyapunov exponent of a periodic orbit $\beta$, by $\Gamma_\delta\subset\Gamma$ 
the subset of elements whose associated periodic orbit has positive Lyapunov exponents greater than $\delta$, and by $\nu_\beta$ the 
invariant   measure of mass $\ell(\beta)$ supported by the periodic orbit $\beta\in\Pi$. 
  
Following a strategy in \cite{Pol:96}, we obtain the following result. 

\begin{theorem}\label{large-deviations} 
	Let $M$ be a compact rank one  surface. 
If $\varphi\colon T^1M\to\bR$ is H\"older continuous and such that $\varphi|_\cH$ is constant and satisfies
\begin{equation}\label{propotpot}
	\alpha(\varphi)\eqdef
	P_{\cM}(\varphi)-\max_{T^1M}\varphi>0,
\end{equation}
then for all $\delta\in(0,\alpha(\varphi))$, we have 
\begin{enumerate}
\item 
	$\displaystyle \limsup_{T\to +\infty}\frac{1}{T}\log
 		\sum_{\beta\in\Pi(T-1,T),\,\chi(\beta)>\alpha(\varphi)-\delta} 
		e^{\int\varphi\,d\nu_\beta}=P_{\cM}(\varphi)$, 
\item for all $x\in\widetilde{M}$, we have 
\[
	\limsup_{T\to +\infty}\frac{1}{T}\log
 \sum_{\gamma\in \Gamma_{\alpha-\delta}, \,T-1\le d(x,\gamma x)< T} e^{\int_x^{\gamma x}\widetilde{\varphi}}=P_{\cM}(\varphi).
\]	 
\end{enumerate}
\end{theorem}
In Section~\ref{examples}, we develop an example which contradicts assumptions and conclusions of this theorem.  

After some preliminaries in Section~\ref{geometric-preliminaries}, we prove Theorem~\ref{equality-pressure-surfaces}    in Section~\ref{pressure}. In Section~\ref{section-equilibrium-states} we deal with equilibrium states. Section~\ref{sec:potentialscon} 
studies the conditions for the potentials used in our main theorems. Theorem \ref{large-deviations} is shown in  Section~\ref{deviations}.


\section{Geometric preliminaries}\label{geometric-preliminaries}

More details for the material in this section can be found in Ballmann 
\cite{Bal:95}, Knieper \cite{Kni:02}, and Eberlein~\cite{Ebe:01}. 

\subsection{Periodic orbits of the geodesic flow}\label{periodic-orbits}

Let $M$ be a smooth compact nonpositively curved manifold. 
The rank of a vector $v\in T^1M$ is the dimension of the set of parallel 
Jacobi fields along the geodesic defined by $v$. It is at least one.  
If this geodesic bounds a euclidean (flat) strip isometric to $[0,\varepsilon]\times\bR$ for some $\varepsilon>0$, 
then this rank is at least $2$. 
On the other hand, vectors whose geodesic eventually enters the negatively curved part 
of the manifold   have rank one.  
The manifold $M$ is a \emph{rank one manifold} if it admits at least one rank one vector.
 The geodesic flow acts on the unit tangent bundle $T^1M$.
 
We denote by $\Gamma=\pi_1(M)$ the fundamental group of $M$ and by $\widetilde{M}$ 
its universal cover, so that $M$ identifies with $\widetilde{M}/\Gamma$.
The set of conjugacy classes of   elements $\gamma\in\Gamma$ is in one-to-one 
correspondence with the set of   free homotopy classes of (oriented) loops in $M$. 
In such a homotopy class, there exists a closed geodesic on $M$ which is length-minimizing. 
When such a   geodesic is regular, it is unique. 
Otherwise, such minimizing geodesics are all parallel and with same length, and 
we choose one closed geodesic among them. 

Each oriented closed geodesic (chosen as above in the singular case) 
lifts to $T^1M$ in a unique way into  a  periodic orbit of 
the geodesic flow.  Let $\Pi$ denote the subset  of primitive periodic orbits 
 (that is, not an iterate of another periodic orbit).  
For all $\beta\in\Pi$, we denote by $\ell(\beta)$ its period 
(that is the length of the associated closed geodesic in $M$)  and $\nu_\beta$  
the Lebesgue measure along the orbit $\beta$. 
Denote by $\Pi(T)\subset \Pi$ the subset of orbits of period smaller than  $T$ 
and by $\Pi(T_1,T_2)$ the subset of orbits whose period belongs to the interval $(T_1,T_2)$.  
In a similar way, denote by $\Pi_\cR, \Pi_\cR(T), \Pi_\cR(T,T')$  (resp. $\Pi_\cH, \Pi_\cH(T), \Pi_\cH(T,T')$) 
the corresponding subsets of regular (resp. singular) periodic orbits.

In the case of zero potentials, Knieper \cite[(Theorem 1.1, Corollary 1.2, 
 and Proposition 6.3)]{Kni:98}  proved the following result. 
 
\begin{theorem}[Knieper~\cite{Kni:98}] \label{Knieper-thm}
Let $M$ be a smooth compact rank one manifold. %
There exists $\varepsilon>0$, such that for $T$ large enough, 
\begin{equation}\label{different-exponential-growth-rates} 
  \#\Pi_\cR(T)> e^{\varepsilon T}\#\Pi_{\cH}(T)\ge 0\,. 
\end{equation}
Moreover,
\begin{equation}\label{eq:christina}
	\lim_{T\to\infty}\frac1T\log\#\Pi_{\cR}(T)=h_{top}(T^1M)\,.
	\end{equation}
When $M$ is a surface, then $\limsup_{T\to +\infty}\frac{1}{T}\log \#\Pi_{\cH}(T)= 0$.
\end{theorem}

The last statement in Theorem~\ref{Knieper-thm} is not true in higher  
dimensions as Gromov~\cite{Gro:78} provides an example of a compact rank one manifold 
in which the number of closed singular geodesics grows exponentially. 
Theorem~\ref{Knieper-thm} implies 
$\limsup_{T\to +\infty}\frac{1}{T}\log \#\Pi(T)>0$, whence 
\begin{equation}\label{TT1}
	\limsup_{T\to +\infty}\frac{1}{T}\log \#\Pi(T)
	=\limsup_{T\to +\infty}\frac{1}{T}\log \#\Pi(T-1,T)\,.
\end{equation}
Corresponding equalities are true for $\Pi_\cR$  
as well as $\Pi_\cH$ 
instead of $\Pi$. \\

In the proof of Theorem \ref{large-deviations}, we will need the following classical observations in nonpositive curvature. 
Any two unit speed geodesics $\beta_1,\beta_2\colon[0,T]\to\widetilde M$ satisfy for $0\le t \le T$
\[	
	d(\beta_1(t),\beta_2(t))\le d(\beta_1(0),\beta_2(0))+d(\beta_1(T),\beta_2(T))\,.
\]	 
Moreover, for every $\eta>0$ there exist $\rho>0$ and $T_0>0$ such that 
for any two unit speed geodesics $\beta_1,\beta_2\colon [0,T]\to\widetilde M$ with $T>T_0$ 
 satisfying $d(\dot\beta_1(0),\dot\beta_2(0))$, $d(\dot\beta_1(T),\dot\beta_2(T))\le\rho$
 we have $d(\dot\beta_1(t),\dot\beta_2(t))\le\eta$ for all $t\in[0,T]$, where we denote by $d$ the Sasaki distance on $T^1M$ (see next section). 
Further, for every $x\in M$ there is exactly one geodesic arc $\beta_\gamma\colon [0,T]\to M$, 
for some $T>0$, with $\beta_\gamma(0)=x=\beta_\gamma(T)$ in each (non-trivial) homotopy class $\gamma\in\Gamma$.


\subsection{Stable and unstable bundles, Lyapunov exponents and the potential $\varphi^{(u)}$}\label{lyapunov}

We refer particularly to \cite[Chapter IV]{Bal:95} for this subsection. 
 Given a vector $v\in T_pM$, we identify $T_vTM$ with  $T_pM\oplus T_pM$ via the isomorphism
\[
	\Psi\colon\xi\mapsto (d\pi(\xi),C(\xi)),
\]
where $\pi\colon TM\to M$ denotes the canonical projection and $C\colon TTM\to TM$ denotes 
the connection map defined by the Levi Civita connection.
Under this isomorphism we have $T_vT^1M\simeq T_pM\oplus v^\perp$, 
where $v^\perp$ is the subspace of $T_pM$ orthogonal to $v$. 
The vector field that generates the geodesic flow is $V\colon v \mapsto (v,0)$.
The Riemannian metric on $M$ lifts to the {\em Sasaki metric} on $TM$ defined by
\[
 \llangle \xi,\eta\rrangle_v
 = \langle d\pi_v (\xi),d\pi _v(\eta)\rangle_{\pi(v)}
 	+ \langle C_v(\xi),C_v(\eta)\rangle_{\pi(v)}.
 \]
Roughly speaking, two vectors in $T^1M$ are close with respect to  the distance induced by the Sasaki metric 
if their  orbits under the geodesic flow stay close during a fixed interval of time. 

A \emph{Jacobi field}  $J$ (see for example \cite[III.C]{GHL:93}) along a geodesic $\gamma$ is a vector field along $\gamma$  which satisfies the Jacobi equation
\begin{equation}\label{e.j}
 	J''(t) +R(J(t),\dot\gamma(t))\dot\gamma(t) = 0,
\end{equation}
where $R$ denotes the Riemannian curvature tensor of $M$ and $'$ denotes
 covariant differentiation along $\gamma$.  

A Jacobi field $J$ along a geodesic $\gamma$ with $\dot\gamma(0)=v$  is uniquely 
determined by its initial conditions $(J(0), J'(0))\in T_pM\times T_pM$, for $p=\pi(v)$. 
Moreover, the set of Jacobi fields along a geodesic $\gamma$ such that $J(0)$ and $J'(0)$ are orthogonal to $v=\dot\gamma(0)$
 is exactly the set of Jacobi fields such that for all $t$, $J(t)$ is normal to $\dot\gamma(t)$ \cite[Theorem 3.43]{GHL:93}. 
They are called the {\em orthogonal Jacobi fields}. 

Given a vector $v\in TM$, a Jacobi field along the geodesic determined by 
$v$ is uniquely determined by $(J(0),J'(0))\in T_\pi(v)M\oplus T_\pi(v)M\simeq T_vTM$. 
An orthogonal Jacobi field along the geodesic determined by $v$ is uniquely determined by 
$(J(0),J'(0))\in v^\perp\oplus v^\perp \subset T_\pi(v)M\oplus  v^\perp\simeq T_vT^1M$.
Therefore, the set of orthogonal Jacobi fields can be identified with the subbundle of $TT^1M$ 
 whose fiber over $v$ is $v^\perp\oplus v^\perp\in T_pM\oplus T_pM\simeq T_vT^1M$. This fiber is 
the orthogonal 
complement in $T_vT^1M\simeq T_pM\oplus v^\perp$ of the subspace spanned by the vector field $(v,0)$ that generates the geodesic flow.

Jacobi fields give a geometric description of the derivative of the geodesic flow.
Given $\xi\in v^\perp\oplus v^\perp\subset  T_v T^1M$, denote by $J_\xi$ the unique Jacobi field along $\gamma_v$ 
with initial conditions $J_\xi(0)=d\pi_v(\xi)$ and $J_\xi'(0)=C_v(\xi)$, or equivalently $(J_\xi(0),J_\xi'(0))=\Psi(\xi)$.  
Then,  $\Psi(dg^t_v(\xi))$ equals   $(J_\xi(t),J_\xi'(t))$.

Orthogonal stable (unstable) Jacobi fields provide a convenient geometric way of describing the vector
 bundles that by Oseledec theorem correspond to non-positive (non-negative) Lyapunov exponents of the geodesic flow on the unit tangent bundle. 
As curvature is nonpositive, the function $t\mapsto \lVert J(t)\rVert$ is convex~\cite[IV, Lemma 2.3]{Bal:95}.
 An (orthogonal)  Jacobi field $J$ along a geodesic is called {\em stable} 
(resp.~{\em unstable}) if $\lVert J(t)\rVert$ is bounded for all $t \geq 0$ (resp.  bounded for all $t \leq 0$). 
Let $J^s$ (resp. $J^u$) denote the set of 
stable (resp. unstable) orthogonal Jacobi fields  and introduce the subspaces
\[
	F^s_v\eqdef\{\xi\in T_vT^1M\colon J_\xi\in J^s\},\quad
	F^u_v\eqdef\{\xi\in T_vT^1M\colon J_\xi\in J^u\}\,.	
\]
Each such subspace in $v^\perp\times v^\perp$ has dimension $n-1$. 
The distributions $F^s\colon v\in T^1M\mapsto F^s_v\subset T_vT^1M$ and 
$F^u\colon v\in T^1M\mapsto F_v^u\subset T_v T^1M$ obtained in this way 
are invariant and continuous (but rarely have higher regularity). 
The subbundle $F^s_v$
 (resp. $F^u_v$) coincides with the space of vectors $\xi\in v^\perp\times v^\perp\subset T_vT^1M$ 
such that $\lVert dg^t_v(\xi)\rVert$ is uniformly bounded for all $t\ge0$ (resp. bounded for all $t\le0$). 

A vector  $\xi$ belongs to $F^s_v\cap F^u_v$  if and only if $t\mapsto \|J_\xi(t)\|$
is constant (as a convex bounded map), or in other words iff
  the function $t\mapsto \lVert dg^t_v(\xi)\rVert$ is constant. 
One says in this case that $J_\xi$ is a {\em parallel Jacobi field} along $\gamma_v$.
When  $M$ is a surface, then $F^s_v\cap F^u_v$ is nontrivial
 if and only if $F^s_v=F^u_v$, that is if and only if the sectional curvature along $\gamma_v$ is everywhere zero.
In general, both subbundles will have nonzero intersection at some vectors $v\in T^1M$.
 In fact, the geodesic flow is Anosov precisely if and only if 
the intersection is zero at \emph{every} vector~\cite{Ebe:73}.

Orthogonal Jacobi fields provide a \emph{continuous} vector bundle 
that defines the following \emph{continuous} potential which is of great importance for many thermodynamic properties of the flow.
Let $\Jac(dg^t|_{F^u_v})$ be the Jacobian of the linear map $dg^t_v\colon F^u_v\to F^u_{g^t(v)}$ 
and consider the geometric potential defined by
\begin{equation}\label{phidef}
	\varphi^{(u)}(v)
	\eqdef -\frac{d}{dt}\Jac(dg^t|_{F^u_v})|_{t=0} 
	= -\lim_{t\to0}\frac1t\log\Jac(dg^t|_{F^u_v})\,,
\end{equation}
which is well-defined and depends differentiably on $F^u_v$ and hence continuously on $v$. 
Moreover, in restriction to  each basic set, the map $v\to F^u_v$ is H\"older continuous,
and hence the restriction of $\varphi^{(u)}$ to such a basic set is also H\"older continuous. 
Indeed, since in the uniformly hyperbolic case $F^s_v$ (resp. $F^u_v$)  coincides with the stable (unstable) 
subspace in the hyperbolic splitting of the tangent bundle, by~\cite[Theorem 19.1.6]{KH:} these spaces vary H\"older continuously in $v$.
  Further, $\varphi^{(u)}$ vanishes on $\cH$ because the norm of any unstable 
Jacobi field is constant along geodesics  in $\cH$.

The Lyapunov exponents of the geodesic flow are well defined for all {\em Lyapunov regular} vectors $v$. 
The set of Lyapunov regular vectors is of full measure with respect to   any invariant probability measure 
 (see for example the appendix of \cite{KH:}). 
For such a {\em Lyapunov regular} vector $v$, classical computations give
$$
 \sum_{i\colon\lambda_i(v)\ge0}\lambda_i(v)
 = \lim_{T \to \infty}  \frac1T\log\,\Jac( dg^T|_{F^u_v})
 = \lim_{T \to \infty} - \frac1T \int_0^T \varphi^{(u)}(g^t(v))\,dt.
$$
Ruelle's inequality \cite{Rue:78} asserts that for all invariant probability measures $\mu$, 
we have 
$$
h(\mu)\le \int_{T^1M}\sum_{i\colon\lambda_i(v)\ge0}\lambda_i(v)\,d\mu\,,
$$
where $h(\mu)$ is the entropy of the measure $\mu$ with respect to the time one of the geodesic flow. 
With the above and~\eqref{def-variational-pressure}, it
 ensures that for all invariant probability measures $\mu\in\cM$  we have 
 \[
 	P_{\cM}(\varphi^{(u)})
	= \sup_{\mu\in\cM}\left(h(\mu)-\int_{T^1M}\sum_{i:\lambda_i(v)\ge 0}\lambda_i(v)\,d\mu\,\right)\le 0.
\]	 
Let $\widetilde m$  be the restriction of the Liouville measure to the invariant open set $\cR$  normalized into a probability measure. 
 In all known examples $\widetilde m$ coincides with the Liouville measure, but this has not been proved in general.  
Ergodicity of $\widetilde m$ was proved in~\cite{Pes:77}. 
By Pesin's formula~\cite{Pes2:77} $h(\widetilde m)= -\int_{T^1M}\varphi^{(u)}\,d\widetilde m$, 
so that 
\begin{equation}\label{pesinn}
	P_{\cM}(\varphi^{(u)})=0,
\end{equation}	 
and the Liouville measure $\widetilde m$ is an equilibrium state for $\varphi^{(u)}$. 
Observe however that when the manifold has flat strips,
  the potential $\varphi^{(u)}$ vanishes on the flat strips. 
In particular, if there are periodic orbits in the flat strips, 
the normalized Lebesgue measure on any such higher rank periodic orbit is also an equilibrium state for $\varphi^{(u)}$.

Finally, in the particular case when $M$ is a surface, 
there exists at most one positive Lyapunov exponent $\chi(v)=-\varphi^{(u)}(v)$.


\section{Pressure and equilibrium states}\label{pressure}

We introduce first several notions of pressure as well as the Poincar\'e series; in Section~\ref{sec:proofT1} 
we prove Theorem \ref{equality-pressure-surfaces}. 
In this section, let $\varphi\colon T^1M\to\bR$ be a continuous map.

\subsection{Variational and topological pressure}\label{sec:3.1}

The \emph{variational pressure} of $\varphi$ (with respect to the flow)
 was defined in~\eqref{def-variational-pressure}  as 
\[
	P_{\cM}(\varphi)
	=\sup_{\mu\in\cM} \left(h(\mu)+\int_{T^1M}\varphi\,d\mu\right).
\]

For $\varepsilon>0$ and $T>0$, a set $E\subset T^1M$ is  
 \emph{$(\varepsilon,T)$-separated} if for all $v,w\in E$, $v\ne w$,
 we have $\max_{0\le t\le T} d(g^t v, g^tw)\ge \varepsilon$. 
Define  
$$
	Z^{sep}_\varphi(T,\varepsilon)
	=\sup_E \sum_{v\in E} e^{\int_0^T\varphi(g^t v)\,dt}\,,
$$
where the supremum is taken over all $(\varepsilon,T)$-separated sets $E$.
The {\em topological pressure} (which, in fact, should rather be called the \emph{metric pressure}) 
of $\varphi$ with respect to the geodesic flow is defined to be the following limit
$$
	P_{top}(\varphi)
	\eqdef\lim_{\varepsilon\to 0}\limsup_{T\to +\infty}\frac{1}{T}\log Z^{sep}_\varphi(T,\varepsilon)
	\,.
$$
Bowen and Ruelle~\cite{BowRue:75} observed  that this definition is equivalent to defining
$P_{top}(\varphi )$ as the topological pressure of the function $\varphi^1\colon v \mapsto \int_0^1\varphi(g^t(v))\,dt$
 with respect to the time-$1$ map $g^1$ of the flow (see~\cite{Wal:81}). 
The variational principle \cite[Theorem 9.10]{Wal:81} ensures that 
$$
	P_{top}(\varphi)=
	P_{\cM}(\varphi)\,.
$$

The \emph{topological entropy} of the geodesic flow, denoted for simplicity   $h=h_{top}(T^1M)$,
 is the topological pressure of the potential $\varphi=0$.
The topological pressure (resp. entropy) of the geodesic flow restricted to any compact set $\Lambda\subset T^1M$
is denoted by $P_{top}(\varphi,\Lambda)$ (resp. $h_{top}(\Lambda)$).  


\subsection{Gurevic Pressure}\label{sec:3.2}

The \emph{Gurevic pressure}, or \emph{periodic orbit pressure}, is defined   as  follows: 
\begin{equation}\label{def:Gur}
	P_{Gur}(\varphi)
	\eqdef\limsup_{T\to +\infty}\frac{1}{T}\log \sum_{\beta\in \Pi(T-1,T)}
		e^{\int\varphi\,d\nu_\beta}\,.
\end{equation}
Recall that $\Pi(T-1,T)$ is the set of {\em primitive periodic orbits} of length between $T-1$ and $T$. 
In a similar way, we define the {\em regular Gurevic pressure} as  
\begin{equation}\label{def:Gur-R}
	P_{Gur,\cR}(\varphi)
	\eqdef\limsup_{T\to +\infty}\frac{1}{T}\log \sum_{\beta\in \Pi_\cR(T-1,T)}
		e^{\int\varphi\,d\nu_\beta}\,.
\end{equation}
It is clear that $P_{Gur,\cR}(\varphi)\le P_{Gur}(\varphi)$. 
Knieper~\cite{Kni:98} proved the following fact. 

\begin{lemma}\label{lem:injinj}
	For all $\varepsilon>0$ smaller than the injectivity radius of $M$ the set $\Pi(T)$ is $(\varepsilon,T)$-separated.
\end{lemma} 
For any continuous potential $\varphi$, we deduce that
$$
P_{Gur}(\varphi)\le P_{top}(\varphi)\,.
$$  
When $\varphi$ is the constant function $\varphi=0$, Theorem~\ref{Knieper-thm} gives
\begin{equation}\label{besicov}
	P_{Gur,\cR}(0)= P_{Gur}(0)=P_{top}(0)=h.
\end{equation}

\subsection{Pressure on basic sets}\label{sec:3.3}

 A flow-invariant set $\Lambda$ is \emph{hyperbolic} if the tangent bundle restricted 
to $\Lambda$ can be written as the Whitney sum of  $dg^t$-invariant subbundles 
$T_\Lambda T^1M=E^s\oplus E\oplus E^u$ where $E$ is the one-dimensional bundle 
tangent to the flow, and there are positive constants $c,\alpha$ such that 
$\lVert dg^t_v(\xi)\rVert\le ce^{-\alpha t}\lVert\xi\rVert$ for $\xi\in E^s_v$, $t\ge0$ 
and $\lVert dg^{-t}_v(\xi)\rVert\le ce^{-\alpha t}\lVert\xi\rVert$ for $\xi\in E^u_v$, $t\ge0$. 
Such a set $\Lambda\subset  T^1M$ is  \emph{locally maximal} if there 
exists a neighborhood $U$ of $\Lambda$ such that $\Lambda=\bigcap_{t\in\bR}{g}^t(\overline U)$. 
The flow ${G}|_\Lambda$ is \emph{topologically transitive} if for all nonempty open sets $U$ and $V$ intersecting $\Lambda$ 
there exists $t\in\bR$ such that ${g}^t(U)\cap V\cap \Lambda\ne\emptyset$.
A {\em basic set} is a compact locally maximal hyperbolic set on which the flow is transitive.

In restriction to any basic set $\Lambda$, in restriction to which the geodesic flow is topologically mixing, 
for $\varphi\colon\Lambda\to\bR$ H\"older continuous, all above introduced pressures of the flow
coincide.
The proof of this classical fact (see~\cite[18.5.1 and 20.3.3]{KH:}  for diffeomorphisms 
and~\cite[Lemma 2.8]{Fra:77}  for flows) uses the {\em specification property}. 
This property holds for geodesic flows of compact negatively curved manifolds~\cite{Ebe:73trans}.
On rank one manifolds, in restriction to any basic set $\Lambda$, by~\cite[(3.2)]{Bow:72b} 
exactly one of the following two distinct cases is true:  
(a) $g_{|\Lambda}$ is a time $\tau$-suspension of an axiom A$\ast$ homeomorphism or, 
(b) $\Lambda$ is C-dense (the unstable manifold of every periodic point in $\Lambda$ is dense in $\Lambda$). 
In case (a) the suspended homeomorphism is topologically mixing and verifies the corresponding pressure 
identities and the pressure of the suspension flow with \emph{constant} time ceiling function does 
not alter these identities for the flow. 
In case (b), by~\cite[(3.8)]{Bow:72b} the flow satisfies the specification property.
We will then denote this quantity shortly by 
\begin{equation}\label{eq:allsame}
	P(\varphi,\Lambda)
	\eqdef P_{Gur}(\varphi,\Lambda)
	= P_{top}(\varphi,\Lambda)
	= P_{\cM}(\varphi,\Lambda)
\end{equation}
Observe further that $P(\varphi,\Lambda_1)\le P(\varphi,\Lambda_2)$ if $\Lambda_1\subset \Lambda_2$. 

When $M$ is a rank one {\em surface}, by~\cite[Theorem 1.5]{BurGel:14} 
there is a family of basic sets 
$\Lambda_1\subset\Lambda_2\subset\cdots\subset \cR$ such that 
 $\bigcup_k\Lambda_k$ is dense in $T^1M$ 
and that for any basic set $\Lambda\subset\cR$, $\Lambda\ne\cR$ 
there exists $k\ge1$ such that $\Lambda\subset \Lambda_k$. 
Moreover, by construction, any regular periodic orbit is 
eventually contained in $\Lambda_k$ for $k$ large enough.
Observe however that the singular periodic orbits
are not contained in these $\Lambda_k$, so that
 $P_{Gur}(\varphi,\Lambda_k)=P_{Gur,\cR}(\varphi,\Lambda_k)$.  
Finally, by~\cite[Theorem 6.2]{BurGel:14}, if  
$\cH\ne\emptyset$ and $\varphi|_{\cH}$ is constant then
we have
\begin{equation}\label{ea:limits}
		P_{top}(\varphi) 
		= \sup_{k\ge 1}P(\varphi,\Lambda_k)
		= \lim_{k\to\infty}P(\varphi,\Lambda_k)
\end{equation}



\subsection{Critical Exponent of Poincar\'e series}\label{sec:3.4}

For a given $x\in\widetilde{M}$, and $\gamma\in\Gamma$, denote by $\int_x^{\gamma x}\widetilde{\varphi}$ the integral of the $\Gamma$-invariant 
lift $\widetilde{\varphi}$ of $\varphi$ to $T^1\widetilde{M}$ 
along the unique lift to $T^1\widetilde{M}$ of the geodesic segment  joining $x$ to $\gamma x$. 
The \emph{Poincar\'e series} associated to $\Gamma$, $\varphi$, $x$, and  $s\in\mathbb{R}$ is  defined by
$$
	Z_{\Gamma,\varphi,x,s}
	\eqdef\sum_{\gamma\in\Gamma}
	e^{\int_x^{\gamma x}\widetilde{\varphi}-s d(x, \gamma x)}\,.
$$
The {\em critical exponent} $\delta_{\Gamma,\varphi,x}$ is defined by the fact that the series diverges when $s<\delta_{\Gamma,\varphi,x}$ and
converges when $s>\delta_{\Gamma,\varphi,x}$. 
Define the sequence $(a_n(x))_{n\in\bN}$ as
\[
	a_n(x)
	\eqdef\sum_{\gamma\in\Gamma,\,n-1\le d(x,\gamma x)<n}
		e^{\int_{x}^{\gamma x}\widetilde{\varphi}}\,.
\]	 
An elementary computation shows that the series $Z_{\Gamma,\varphi,x,s}$ converges (diverges) if, and only if, 
the series $\sum_{n\in\bN} a_n(x)e^{-sn}$ converges (diverges), so that 
$$
\delta_{\Gamma,\varphi,x}
=\limsup_{n\to +\infty}\frac{1}{n}\log \sum_{\gamma\in\Gamma,\,n-1\le d(x,\gamma x)<n}e^{\int_{x}^{\gamma x}\widetilde{\varphi}}\,.
$$
A simple computation shows that when  $\delta_{\Gamma,\varphi,x}>0$, it also satisfies
\begin{equation}\label{eq:deltaaa}
	\delta_{\Gamma,\varphi,x}
	=\limsup_{n\to +\infty}\frac{1}{n}\log 
		\sum_{\gamma\in\Gamma,  d(x,\gamma x)\le n}
		e^{\int_{x}^{\gamma x}\widetilde{\varphi}}\,.
\end{equation}

\begin{remark}
\rm Contrarily to the negative curvature case,
 the critical exponent $\delta_{\Gamma,\varphi,x}$ does depend  on the point $x$, and it does not 
seem possible to remove this dependance by an elementary reasoning. 
Observe however the following facts:\\
$\bullet$ For all $x\in \widetilde{M}$ we have $\lvert\delta_{\Gamma,\varphi,x}-\delta_{\Gamma,0,x}\rvert\le \lVert\varphi\rVert_\infty$. \\
$\bullet$ $\delta_{\Gamma,0,x}$ is independent of $x$ and coincides with the topological entropy of the geodesic flow, which is finite.\\
$\bullet$ The map $x\mapsto \delta_{\Gamma,\varphi,x}$ is 
$\Gamma$-invariant. When  $\varphi$ is continuous,  the map
$x\mapsto \delta_{\Gamma,\varphi,x}$ is moreover continuous on $\widetilde{M}$. 
It induces therefore a continuous, and therefore uniformly continuous, map on the quotient manifold $M$. 

The last point is elementary to check, thanks to the 
uniform continuity of the lift $\widetilde{\varphi}$ of $\varphi$ to $T^1\widetilde{M}$, and
to the following geometric fact: in nonpositive curvature, 
 given $\eta>0$ and $x,y$ satisfying $d(y,x)<\eta$,  
 the geodesic segments $[x,\gamma x]$ and $[y,\gamma y]$ stay at distance less than $\eta$.
\end{remark}
By the above remark, the following quantity is well defined and finite
\[
	\delta_{\Gamma,\varphi}
	\eqdef \max_{x\in\widetilde{M}}\delta_{\Gamma,\varphi,x}\,.
\]


\subsection{Examples of compact rank one manifolds}\label{examples}
We want to mention some examples and discuss the hypotheses in our main results.

\begin{example}
A compact connected nonpositively curved manifold on which every geodesic eventually crosses the negatively curved part of the manifold is a compact rank one manifold where $\cH=\emptyset$. By a result of Eberlein~\cite{Ebe:73}, its geodesic flow is Anosov (and
 our results are already well-known in this case). The typical example is a 
 manifold where the curvature is negative everywhere except in a sufficiently small disk where it is equal to zero (Figure~\ref{fig:1}).  
\begin{figure}
\begin{overpic}[scale=.45,
  ]{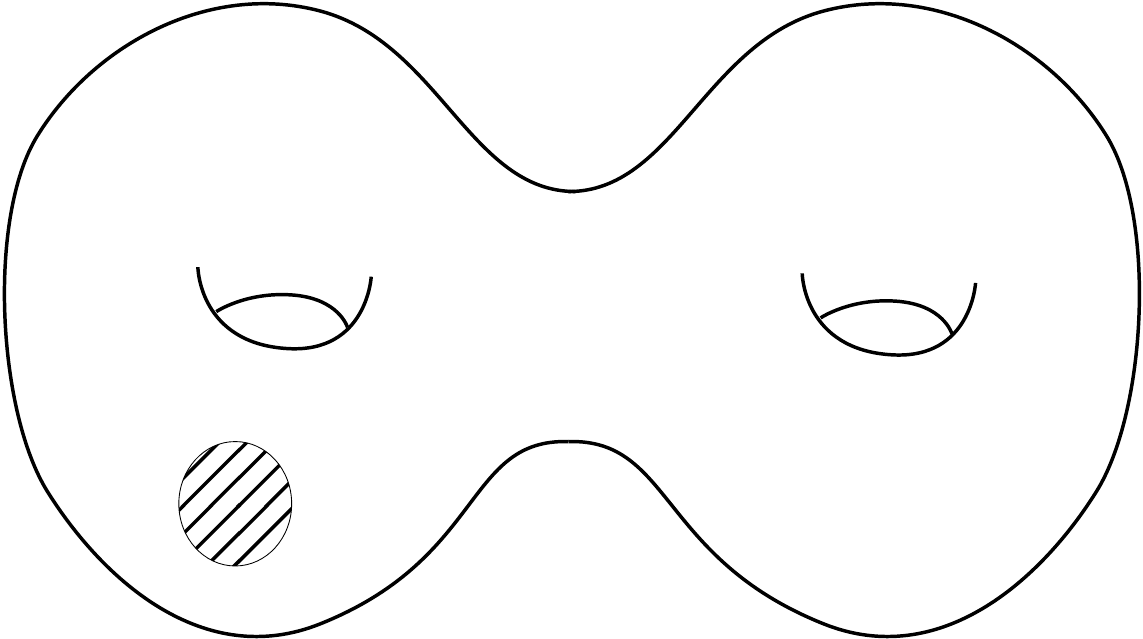}
        \put(6,40){\tiny curvature$<0$}
      \put(27,10){\tiny curvature$=0$}
\end{overpic}
\caption{Nonpositively curved surface with Anosov geodesic flow}
\label{fig:1}
\end{figure}
\end{example}

\begin{example}\label{ex:2}
Second, the simplest example of a compact rank one surface is a compact connected surface $S$ of genus $g\ge 2$ which has a periodic euclidean cylinder of positive width somewhere and has negative curvature elsewhere. 
The set $\mathcal{H}$ of higher rank vectors is the set of vectors whose geodesic stays all the time in the flat cylinder. In the `degenerate' case where the flat cylinder has width equal to zero, the set $\cH$ is reduced to a single periodic orbit (Figure~\ref{fig:2}). 
\begin{figure}
\begin{overpic}[scale=.45,
  ]{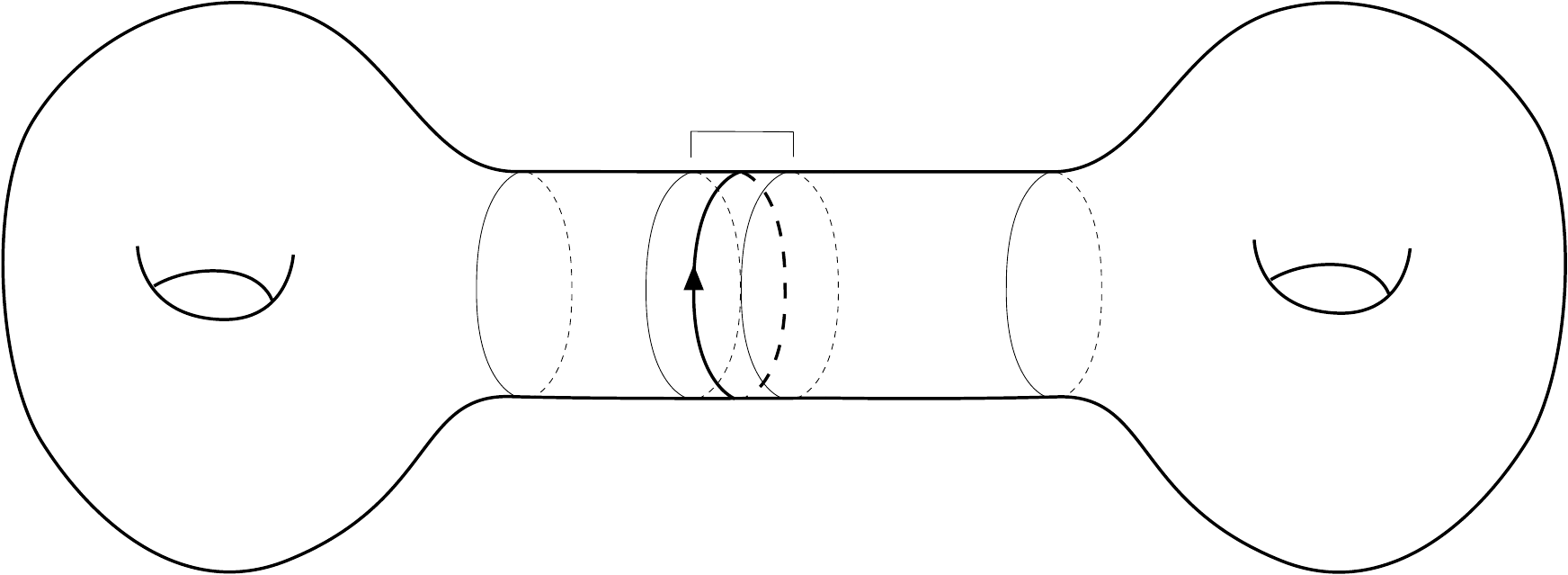}
      \put(45,8){\tiny$\beta_0$}
      \put(46,30){\tiny$\varepsilon_0$}
\end{overpic}
\caption{Surface with periodic euclidean cylinder}
\label{fig:2}
\end{figure}
\end{example}

\begin{example}\label{ex:3}
\smallskip
Finally, let us mention examples introduced by Gromov \cite{Gro:78} and further studied by  Knieper \cite{Kni:98} where $\#\Pi_\cH(T)$ can have exponential growth in $T$.  
First, consider a hyperbolic punctured torus, and   modify the neighbourhood of the puncture so that it becomes isometric to a 
flat cylinder $I\times C_1$, where $I$ is an interval and $C_1$ a circle. Call this punctured and flattened torus $T_1$.
Consider the three dimensional manifold $M_1=T_1\times C_2$, where $C_2$ is another circle. The boundary is the  product (flat) torus $C_1\times C_2$ of two circles. 
Consider another such manifold $M_2=C'_1\times T_2$, where $T_2$ is also a flattened punctured torus, $C'_1$ is isometric to $C_1$ and $C'_2=\partial T_2$ to $C_2$. 
Now glue $M_1$ and $M_2$ along their boundaries by identifying $C_1\times C_2$ with $C'_1\times C'_2$. 
The resulting manifold (compare Figure~\ref{fig:3}) is a compact connected three-dimensional rank one manifold, 
whose singular periodic orbits have exponential growth \cite{Kni:98}.
Indeed, one observes that the product of a periodic orbit in $T_1$ and a point in $C_2$ is a periodic orbit of $M_1$ of rank two. 
Therefore, there is an injective map from the set of periodic orbits of $T_1$ into the set of  singular periodic orbits of $M_1$. 
Moreover, the length of a periodic orbit in $T_1$ is at most the length of the periodic geodesic in the same homotopy class in the hyperbolic punctured torus. 
The exponential growth rate of periodic orbits on the hyperbolic punctured torus implies therefore the exponential growth rate of singular periodic orbits
of $M_1$. 
\begin{figure}
\begin{overpic}[scale=.45,
  ]{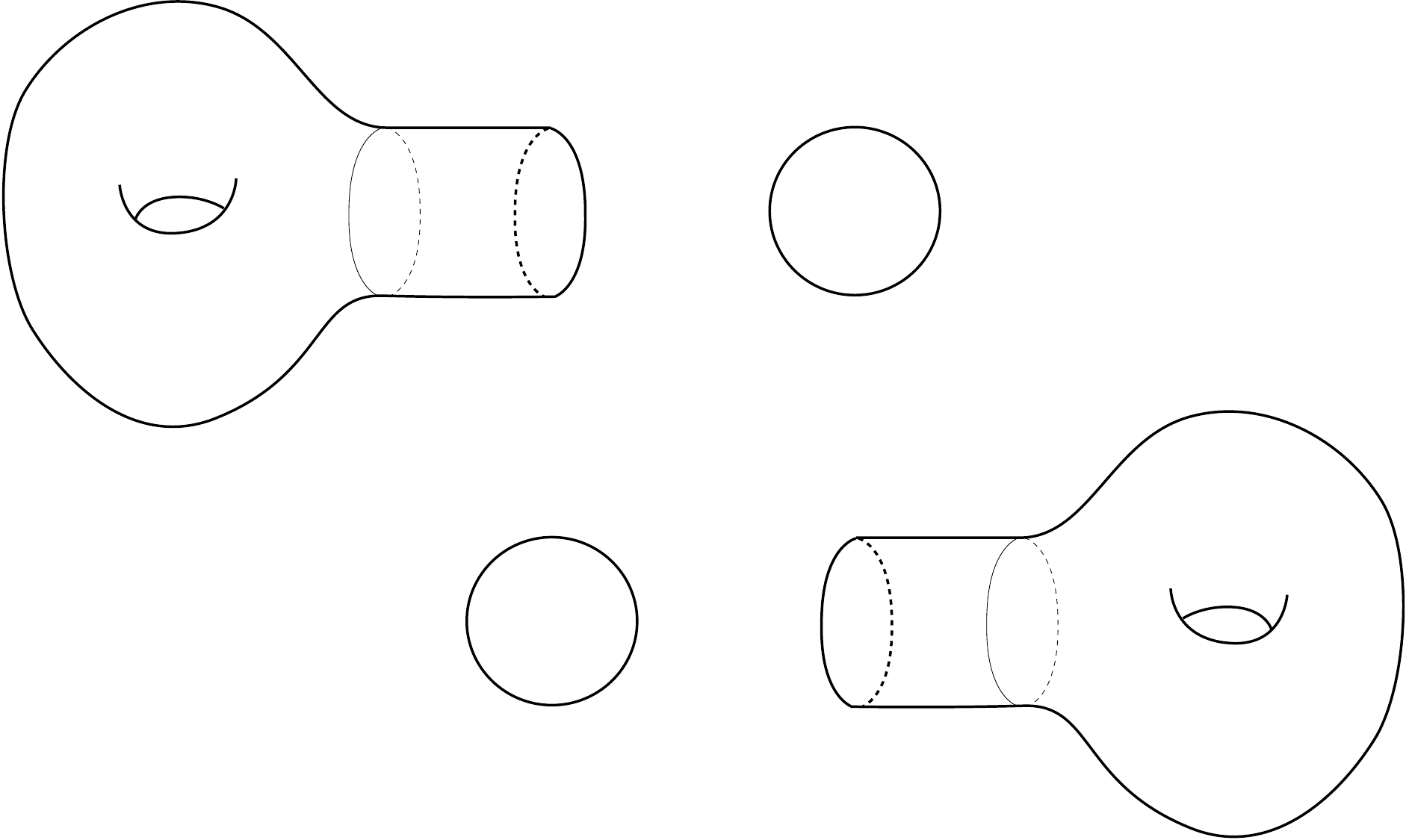}
      \put(-12,45){$M_1$}
      \put(105,14){$M_2$}
      \put(32,53){\tiny$I$}
      \put(38,53){\tiny$C_1$}
      \put(59,53){\tiny$C_2$}
      \put(67,5){\tiny$I$}
      \put(59,5){\tiny$C_2'$}
      \put(38,5){\tiny$C_1'$}
      \put(46,44){$\times$}
      \put(49,15){$\times$}
      \put(38,30){$\updownarrow$}
      \put(59,30){$\updownarrow$}
\end{overpic}
\caption{Gromov's example}
\label{fig:3}
\end{figure}
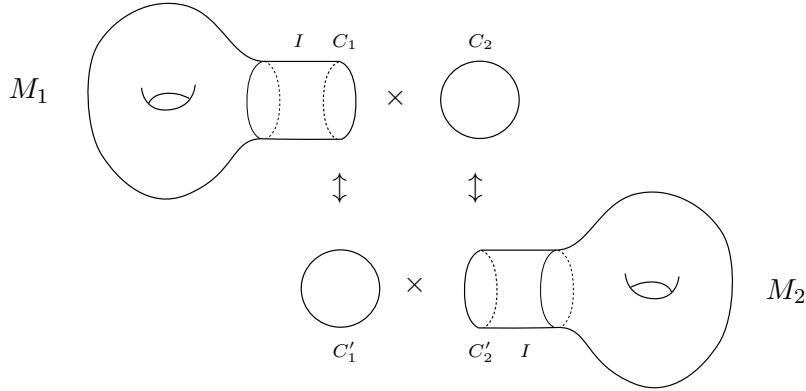
\end{example}

\smallskip
Focusing on the Example~\ref{ex:2}, we would like to discuss the hypotheses of our main results.
Let us develop a little bit more the case where the euclidean cylinder is of positive width, say of width $L$. Denote by $h$ the topological entropy of the flow. Choose a (higher rank)  vector $v_0$  tangent to a periodic geodesic inside the cylinder. 
We will define a potential $\varphi\colon T^1S\to\bR$ that will not satisfy hypotheses and conclusions of Theorem \ref{equality-pressure-surfaces}. 

Choose $\varphi$ to be constant along the orbit $\beta_0$ of $v_0$ and satisfying $\varphi(v_0)=(1+\eta)\,h>h$ for some  $\eta>0$. 
Choose some $\varepsilon>0$ such that $2\varepsilon/L< \eta/(1+\eta)$. 
Let $V_\varepsilon(v_0)$ be the set of unit vectors in $T^1S$ such that the distance from their basepoint to the closed geodesic associated to $v_0$ is less than $\varepsilon/2$, and whose angle with the orbit of $v_0$ is less than $\varepsilon/2$. 
By an elementary euclidean argument already used in \cite{CS:11}, if $v$ is any unit tangent vector whose basepoint is outside the  cylinder, 
the proportion of time spent by the piece of orbit $(g^t(v))_{0\le t\le T}$ inside $V_\varepsilon(v_0)$ is at most $\varepsilon/L$. 
Now extend the potential $\varphi$ continuously to $T^1S$ in such a way that $\varphi\ge 0$ and $\varphi\equiv 0$ outside $V_\varepsilon(v_0)$. 

Observe that this potential $\varphi$ does not satisfy assumption (\ref{con-pot}) of Theorem \ref{equality-pressure-surfaces}.
If $\beta$ is a rank one periodic orbit then the above remark shows that 
\[
	\int \varphi\, d\nu_\beta
	\le \frac{\varepsilon}{L}\lVert\varphi\rVert_\infty\,\ell(\beta)
	=\frac{\varepsilon}{L}\,(1+\eta)h\,\ell(\beta)
	\,.
\]	 
We deduce easily that 
\[
	P_{Gur,\cR}(\varphi)\le \frac{\varepsilon}{L}\|\varphi\|_\infty+h
	<\frac{\eta}{2(1+\eta)}\|\varphi\|_\infty+h=h(\frac{\eta}{2}+1)\,.
\]	
On the other hand, the topological pressure of $\varphi$ certainly satisfies
\[
	P_{top}(\varphi)
	=P_\cM(\varphi)\ge \int \varphi\,d\widehat{\nu}_{\beta_0}
	=h(1+\eta)> P_{Gur,\cR}(\varphi)\,,
\]
where $\widehat{\nu}_{\beta_0}$ is the normalized periodic measure supported by $\beta_0$. 
In particular, it proves that equalities~\eqref{eqtheoremsurface} in Theorem \ref{equality-pressure-surfaces} do not hold. 

Now, let $\mu$ be a flow-invariant measure with support in the set $\cR$ of regular vectors, and consider a generic recurrent vector $v$ based outside the flat cylinder. 
We can find $T$ large enough such that the ergodic average on $(g^t(v))_{0\le t\le T}$ is very close to $\mu$, so that in particular, 
\[
	\left\lvert\int\varphi\,d\mu-\frac{1}{T}\int_0^T\varphi(g^t(v))\,dt\right\rvert
	\le \varepsilon\,.
\]	 
But the same argument as above shows that this piece of orbit $(g^t(v))_{0\le t\le T}$ 
cannot stay more than a proportion $\varepsilon/L$ inside $V_\varepsilon(v_0)$. Therefore, 
  $\int\varphi\,d\mu<\varepsilon+\lVert\varphi\rVert_\infty\varepsilon/L$, so that $h(\mu)+\int\varphi\,d\mu\le  \varepsilon+h(1+\eta/2)$. 
If $\varepsilon$ is small enough, this quantity is bounded from above by $h(1+\eta)$, so that
\[
	P_{top}(\varphi)>\sup_{\mu\in\cM(\cR)}\left( h(\mu)+\int\varphi\,d\mu\right)\,.
\]	
In particular, it implies that $P_{top}(\varphi)=\|\varphi\|_\infty=h(1+\eta)$ and the periodic orbit measure $\widehat\nu_{\beta_0}$ is an equilibrium state for $\varphi$. This equilibrium state is unique when $\varphi$ is chosen to be strictly decreasing in the neighborhood of $v_0$. (In the other case, it could happen that 
other periodic orbits in $V_\varepsilon(v_0)$ are also equilibrium states.)

As our definition of Gurevic pressures involves only primitive periodic orbits, and there is a unique primitive periodic orbit in $\Pi_\cH(T)$, we
see that $P_{Gur}(\varphi)=P_{Gur,\cR}(\varphi)$ even though~\eqref{con-pot} is not satisfied. 
If our definitions of Gurevic pressures were modified to take into account all periodic orbits (not only primitive ones) then the above example would lead to a regular Gurevic pressure strictly less than the full Gurevic pressure. Moreover, the latter would be attained on singular periodic orbits and would be equal to the topological pressure.  

To conclude on this example, observe that $\varphi$ does not satisfy neither the assumption~\eqref{propotpot} of Theorem~\ref{large-deviations}, nor its conclusions.

It would be interesting to adapt Example~\ref{ex:2} with the cylinder to Gromov's Example~\ref{ex:3} to provide an example for the strict inequalities
$P_{Gur,\cR}(\varphi)<P_{Gur}(\varphi)<P_{top}(\varphi)$.

\subsection{Proof of Theorem \ref{equality-pressure-surfaces}}\label{sec:proofT1}

As each of the above quantities $\ccP=P_{\cM}$, $P_{top}$, $P_{Gur},$ 
and $P_{Gur,\cR}$ satisfy $\ccP(\varphi+c)=\ccP(\varphi)+c$ for every $c\in\mathbb{R}$, 
without loss of generality, in the following we can assume that all such pressures are positive. 

By~\cite[Lemma 4]{Pol:96}, assumption $P_{top}(\varphi)>0$, and~\eqref{eq:deltaaa}, 
for all $x\in \widetilde{M}$ we have 
\[
	\delta_{\Gamma,\varphi,x}
	=\limsup_{n\to +\infty}\frac{1}{n}\log 
		\sum_{\gamma\in\Gamma,  d(x,\gamma x)\le n}
		e^{\int_{x}^{\gamma x}\widetilde{\varphi}}
	\le P_{top}(\varphi)
\]

In 
\cite[Proof of Proposition 2, page 161]{Pol:96}, 
Pollicott proves (without stating explicitely) that 
$P_{Gur}(\varphi)\le \sup_{x\in \widetilde{M}}\delta_{\Gamma,\varphi,x}=\delta_{\Gamma,\varphi}$. 
His proof is written for $\varphi$ Lipschitz, but his argument is valid for uniformly continuous potentials. 
As $T^1M$ is compact, any continuous potential is uniformly continuous, so that its arguments apply.  
This together with the above shows the first claim~\eqref{inequality-pressure-manifolds} in Theorem \ref{equality-pressure-surfaces}.

Consider  $\varepsilon_0$ as defined in~\eqref{defesspilon}. 
 Let us   prove that  under the hypothesis~\eqref{con-pot} 
we have $P_{Gur, \cR}(\varphi)\ge P_{Gur}(\varphi)$.  
Up to replacing  $\varphi$ by $\varphi-\inf_{\mu\in\cM(\cR)}\int\varphi\,d\mu$, we can assume that $\inf_{\mu\in\cM(\cR)}\int\varphi\,d\mu = 0$.
The hypothesis then becomes $\max_{\mu\in\cM(\cH)}\int\varphi\,d\mu<\varepsilon_0$. 
As $\int\varphi\,d\mu\ge 0$ for all $\mu\in\mathcal{M}(\cR)$, using~\eqref{besicov} we get
\[
	P_{Gur}(\varphi)\ge P_{Gur, \cR}(\varphi)\ge P_{Gur, \cR}(0)
= h >0\,.
\]  
In particular, all pressures are positive, as required at the beginning of the proof. 
The fact that $P_{Gur,\cR}(\varphi)>0$ then implies that 
\begin{equation}\label{equality-gurevic} 
P_{Gur,\cR}(\varphi)
=\limsup_{T\to +\infty}\frac{1}{T}\log \sum_{\beta\in\Pi_\cR(T-1,T)}e^{\int\varphi\,d\nu_\beta}
= \limsup_{T\to +\infty}\frac{1}{T}\log \sum_{\beta\in\Pi_\cR( T)}e^{\int\varphi\,d\nu_\beta}\,.
\end{equation}
As   $P_{Gur}(\varphi)>0$,  the above equalities also 
hold for $P_{Gur}$ with the sums taken over $\Pi(T-1,T)$ and $\Pi(T)$, respectively. 
In the case $\varphi=0$, together with~\eqref{eq:christina} it gives 
$$
	\limsup_{T\to +\infty}\frac{1}{T}\log \#\Pi_{\cR}(T-1,T)
	=\lim_{T\to +\infty}\frac{1}{T}\log\#\Pi_\cR(T)=h>0\,.
$$ 
Observe that
\begin{eqnarray*}
	\sum_{\beta\in \Pi( T)} e^{\int \varphi\,d\nu_\beta}
	&=&\sum_{\beta\in \Pi_{\cR}( T)} e^{\int \varphi\,d\nu_\beta}
		+ \sum_{\beta\in \Pi_\cH( T)} e^{\int\varphi\,d\nu_\beta}\\
	&\le& \sum_{\beta\in\Pi_{\cR}( T) }e^{\int \varphi\,d\nu_\beta}   
		+ e^{T\max_{\mu\in\cM(\cH)}\int\varphi\,d\mu}\#\Pi_\cH( T)\,.
\end{eqnarray*} 
Choose $0<\delta <  \varepsilon_0-\sup_{\mu\in\cM(\cH)}\int\varphi\,d\mu$. 
By Theorem~\ref{Knieper-thm}, for $T$ sufficiently large 
we have $\#\Pi_\cH( T)\le e^{-T (\varepsilon_0-\delta)}\#\Pi_\cR(T)$.
 As $\inf_{\mu\in\cM(\cR)}\int\varphi\,d\mu = 0$, the above terms can be estimated further by
\begin{eqnarray*}		
	&\le& \sum_{\beta\in \Pi_{\cR}( T) } e^{\int \varphi\,d\nu_\beta}  
		+ e^{T\max_{\mu\in\cM(\cH)}\int\varphi\,d\mu}
			e^{-T(\varepsilon_0-\delta)}\#\Pi_\cR(T) \\
&\le& \sum_{\beta\in \Pi_{\cR}( T) } e^{\int \varphi\,d\nu_\beta}  
		+ e^{T\max_{\mu\in\cM(\cH)}\int\varphi\,d\mu}
			e^{-T(\varepsilon_0-\delta)}\sum_{\beta\in\Pi_\cR(T)}
				e^{\int\varphi\,d\nu_\beta} \\
	&=& \sum_{\beta\in \Pi_{\cR}( T) } e^{\int \varphi\,d\nu_\beta} 
		 \left(1+e^{T(\max_{\mu\in\cM(\cH)}\int\varphi\,d\mu -\varepsilon_0+\delta)}\right) \,.
\end{eqnarray*} 

As   $\max_{\mu\in\cM(\cH)}\int\varphi\,d\mu<\varepsilon_0-\delta$, considering the limsup of $\frac{1}{T}\log$ of the above quantities
leads to  $P_{Gur}(\varphi)\le P_{Gur,\cR}(\varphi)$. This proves~\eqref{equality-Gurevic}. 

Consider now the case where $M$ is a rank one surface, where we can apply results of \cite{BurGel:14}. 
Assume in addition that $\varphi$ is H\"older and that $\varphi|_{\mathcal H}$ is constant.
In restriction to any basic set $\Lambda_k$, \eqref{eq:allsame} holds.
Obviously, periodic orbits of the geodesic flow restricted to $\Lambda_k$
are in $\Pi_{\mathcal{R}}\subset \Pi$. 
Thus, naturally we have 
$P(\varphi,\Lambda_k)\le P_{Gur,\cR}(\varphi)\le 
P_{Gur}(\varphi)\le \sup_{x\in\widetilde{M}} \delta_{\Gamma,\varphi,x}=
\delta_{\Gamma,\varphi} \le P_{top}(\varphi)=P_{\cM}(\varphi)$ by what precedes. 
Further, by~\eqref{ea:limits} this lower bound converges to $P_{top}(\varphi)$.
This proves~\eqref{eqtheoremsurface} and finishes the proof of Theorem \ref{equality-pressure-surfaces}.
\qed


\section{Equilibrium states}\label{section-equilibrium-states}

Given a potential, it is interesting to identify (ergodic) equilibrium measures, as they reflect 
the dynamics weighted by the potential.

In our situation, the existence of equilibrium states follows immediately from the upper semi-continuity~\cite{Wal:81} 
of the entropy map $\mu\mapsto h(\mu)$ on the set of
flow-invariant probability measures since the geodesic flow is smooth~\cite{New:89}. 
By Bowen~\cite[Theorem 3.5]{Bow:72}, 
this upper semi-continuity can be deduced from  the fact that the geodesic flow 
is $h$-expansive~\cite[Proposition 3.3]{Kni:98}.  
 
In general, the uniqueness of equilibrium states can be deduced from 
the existence of a {\em Gibbs measure} (see \cite[Section 20.3]{KH:}).
On the other hand, when the pressure map $\varphi\mapsto P_{\cM}(\varphi)$ is 
differentiable at $\varphi$ in every direction or in a set of directions that is dense 
in the weak topology~\cite[Corollary 3.6.14]{PrzUrb:10}
then  there exists also a unique equilibrium state. 
But we are still not able to apply one of these two strategies to rank one geodesic flows.

Classical arguments now lead us to the following results.  
We will always assume that $\varphi$ is H\"older continuous and $\varphi|_\cH$ is constant.

\begin{remark}\label{equilibrium} 
	Let $M$ be a smooth compact rank one surface and consider an increasing family 
of basic sets $(\Lambda_k)_{k\in\bN}$ as provided in~\cite{BurGel:14}. 
With respect to the restriction of the geodesic flow to $\Lambda_k$, the potential $\varphi$ 
admits a unique equilibrium state that we denote by $\mu_{\varphi,k}$. 
Then any accumulation point (with respect to the weak$\ast$ topology) of the sequence of 
measures $(\mu_{\varphi,k})_{k\in\bN}$ is an equilibrium state for the potential $\varphi$ (with respect to the flow on $T^1M$). 
\end{remark} 

\begin{remark}\label{periodic-orbit-converge-to-equilibrium} 
	Given $\beta\in\Pi(T)$, denote by $\nu_\beta$ is the Lebesgue measure on the 
periodic orbit $\beta$ and $\widehat\nu_\beta=\ell(\beta)^{-1}\nu_\beta$ the normalized (probability) measure. 
	In   \cite[Theorem 9.10]{Wal:81}, 
Walters shows that the accumulation points of weighted averages 
of Dirac measures on $(\varepsilon,T)$-separated sets that approximate
 well the topological pressure are equilibrium measures for $\varphi$. 
By Lemma~\ref{lem:injinj}, for every sufficiently small $\varepsilon>0$ 
and for all $T>0$, the set $\Pi(T-1,T)$ of periodic orbits of length approximately $T$ 
is $(\varepsilon,T)$-separated. Thus, by Theorem \ref{equality-pressure-surfaces}
 when $M$ is a smooth rank one {\em surface}, these sets $\Pi(T-1,T)$ allow to approximate the topological pressure. 
In this situation, any accumulation point of 
the following weighted averages on periodic orbits 
$$
\frac{\sum_{\beta\in\Pi(T-1,T)}e^{\int \varphi\,d\nu_\beta}
\,\widehat\nu_\beta}{\sum_{\beta\in\Pi(T-1,T)}e^{\int
\varphi\,d\nu_\beta}},
$$
is an equilibrium measure  of $\varphi$ (with respect to the flow on $T^1M$). 	
\end{remark}


\section{Hyperbolic Potentials}\label{sec:potentialscon}

In this section we discuss assumptions~\eqref{con-pot} and~\eqref{propotpot}. 
For $t>0$ define $\varphi^t\colon v\mapsto\int_0^t\varphi(g^\tau(v))\,d\tau$. 
We first show some preliminary result based on classical arguments (see for example~\cite{InoRiv:12}) that we repeat for completeness.

\begin{lemma}
	$\displaystyle
	\max_{\mu\in\cM(\cH)}\int\varphi\,d\mu
	= \lim_{t\to\infty}\max_{v\in \cH}\frac 1t\varphi^t(v)
	= \inf_{t>0}\max_{v\in \cH}\frac 1t\varphi^t(v) $.
\end{lemma}

\begin{proof}
Given $\mu\in\cM$, by flow invariance, for $t>0$  we have
\[
	\frac1t\int_{T^1M} \varphi^t\,d\mu 
	= \frac1t\int_{T^1M} \int_0^t \varphi\circ g^s\,ds\,d\mu
	= \frac1t\int_0^t \int_{T^1M} \varphi\circ g^s\,d\mu\,ds
	=\int_{T^1M}\varphi\,d\mu\,.
\]
Thus
\[
	\int_{T^1M} \varphi\,d\mu \le \max_{\cH}\frac1t \varphi^t\,.
\]
Taking the supremum over  $\mu\in\cM(\cH)$, and then the infimum over   $t>0$ gives
\[
	\max_{\mu\in\cM(\cH)}\int\varphi\,d\mu \le \inf_{t>0}\max_{\cH}\frac1t \varphi^t\,.
\]
It remains to show  the opposite inequality. Given $n\ge1$ choose $v_n$ in the compact 
invariant set $\cH$ such that the function $\frac1n\varphi^n$ attains its maximum in $v_n$. 
Consider the probability measure $\nu_n$ defined by
\[
	\int\psi\,d\nu_n\eqdef\frac1n\int_0^n\psi(g^s(v_n))\,ds
	\quad\text{ for every }\psi\in C^0(T^1M,\mathbb{R})\,.
\]
Choose a subsequence $(n_k)_{k\ge1}$ of positive integers such that 
\[
	\lim_{k\to\infty}\frac{1}{n_k}\varphi^{n_k}(v_{n_k})
	=\limsup_{n\to\infty}\frac1n\varphi^n(v_n)
	=\limsup_{n\to\infty}\max_\cH\frac1n\varphi^n	\,.
\]
Possibly taking a subsequence, the sequence of measures $(\nu_{n_k})_{k\ge1}$ 
converges in the weak$\ast$ topology to an invariant probability measure $\mu$ supported in $\cH$. We obtain
\[
	\int\varphi\,d\mu
	=\lim_{k\to\infty}\int\varphi\,d\nu_{n_k}
	= \lim_{k\to\infty}\frac{1}{n_k}\varphi^{n_k}(v_{n_k})
	=\limsup_{n\to\infty}\max_\cH\frac1n\varphi^n\,.
\]
This together with the above proves the claim.
\end{proof}

The following provides an immediate stronger version of condition~\eqref{con-pot}.

\begin{corollary}
	Any continuous potential $\varphi\colon T^1M\to\bR$ satisfying
	\[
		\max_\cH\frac1t\varphi^t<\inf_\cR\frac1t\varphi^t+\varepsilon_0
	\]
	for some $t>0$ also satisfies condition~\eqref{con-pot}.
\end{corollary}

We now study condition~\eqref{propotpot}.
Given continuous potentials $\varphi\colon T^1M\to\bR$, consider  
\[
	\alpha(\varphi)\eqdef P(\varphi)-\max\varphi^1.
\]
The potential $\varphi$ is said to be \emph{hyperbolic} if there exists   $t>0$ such that 
\[
	 t\,P(\varphi)-\max \varphi^t>0\,.
\]
Two continuous potentials $\varphi,\psi\colon T^1M\to\bR$ are said to be \emph{co-homologous} 
(with respect to the flow) if there exists a continuous function $\eta\colon T^1M\to\bR$ 
 such that $\varphi^t-\psi^t=\eta\circ g^t-\eta$ for every $t$. 
This is equivalent to the fact that $\varphi-\psi=\lim_{t\to0}(\eta\circ g^t-\eta)/t$.  

In the following, we    require that $\alpha(\varphi)>0$. 
This seems a very restrictive hypothesis. However, following~\cite[Proposition 3.1]{InoRiv:12}
 verbatim, in the case of a flow, we get the following equivalences. 

\begin{lemma}
The following facts are equivalent:
\begin{itemize}
		\item $\alpha(\varphi)>0$.
		\item The potential $\varphi$ is hyperbolic.
		\item The metric entropy of each equilibrium state of $\varphi$ is strictly positive.
		\item There exists a continuous potential $\psi$ co-homologous to $\varphi$ such that $\alpha(\psi)>0$.
		\item Every continuous potential co-homologous to $\varphi$ is hyperbolic.
	\end{itemize}
\end{lemma}	

As one of our main interests is the unstable Jacobian $\varphi^{(u)}$ defined in~\eqref{phidef}, we observe the following.

\begin{lemma}\label{lem:11} 
	For all $t<1$ we have $\alpha(t\varphi^{(u)})>0$.
\end{lemma}

\begin{proof}
	Recall first that $t\mapsto P(t\varphi^{(u)})$ is convex and hence continuous. 
For $t=0$, we have $P(0)=h>0$. For $t=1$, Ruelle's inequality implies  $P(\varphi^{(u)})\le 0$, 
Pesin's formula implies $P(\varphi^{(u)})=0$, and the restricted Liouville measure $\widetilde m$  
is an equilibrium state for $\varphi^{(u)}$, recall~\eqref{pesinn}. 
In particular, we have $P(t\varphi^{(u)})\ge t \int\varphi^{(u)}\,d\widetilde m>0$ for all $t<1$.

It is not hard to check that $\varphi^{(u)}\le0$ (see, for example \cite[Lemma 2.4]{BurGel:14}).
 Further, $\min \varphi^{(u)}=\min_{\cR} \varphi^{(u)}<0$. 
Thus, $\max(t\varphi^{(u)})=-|t|\min \varphi^{(u)}>0$ for $t<0$, and $\max(t\varphi^{(u)})=0$ 
for $t\ge 0$. In particular, $t\mapsto P(t\varphi^{(u)})$ is non-increasing.

We deduce that $P(t\varphi^{(u)})-\max(t\varphi^{(u)})=P(t\varphi^{(u)})>0$ for $0\le t <1$ 
and $P(t\varphi^{(u)})-\max(t\varphi^{(u)})\ge P(0)+\lvert t\rvert\min\varphi^{(u)}>P(0)>0$ 
for $t<0$. 
\end{proof}

When $\alpha(\varphi)>0$,   we will see in the next section that the contribution 
of periodic orbits with small positive Lyapunov
 exponent   in the growth rate  of the definition~\eqref{def:Gur} 
of Gurevich pressure is negligible.

\section{Level-2 Large Deviation Principle} \label{deviations}


In this section, we assume that $M$ is a surface, and $\varphi$ is a continuous potential 
such that $\varphi_{|\cH}$ is constant and $\alpha(\varphi)=P_{\mathcal{M}}(\varphi)-\max_{T^1M}\varphi>0$. 
Adding a constant to $\varphi$, we will assume that $\varphi_{|\cH}=0$.
We denote by $P(\varphi)$ the topological pressure, which coincides 
with all other pressures  by Theorem \ref{equality-pressure-surfaces}. 
The assumption $\alpha(\varphi)>0$ with $\varphi_{|\cH}=0$ implies in particular  $P(\varphi)>0$. 

In order to formalize our level-2 large deviation results
\footnote{Level-2 deviation refers to deviations of empirical measures to distinguish from 
so-called level-1 deviations of empirical (Birkhoff) sums or integrals of an  observable, see~\cite{Ell:85}.}, 
let us first introduce a rate
function. Closely related approaches can be found, for example, in~\cite{Pol:95}.
Let $\cP$ be the space of all (not necessarily invariant) Borel 
probability measures on $T^1M$ endowed with the topology of weak$\ast$ convergence, and
$\cM$ is the subspace of invariant measures under the geodesic flow.
Given  $\varphi\in C^0(T^1M,\bR)$, we define $Q_\varphi\colon C^0(T^1M,\bR)\to
\bR$ by
 \[
    Q _\varphi(\psi)\eqdef P(\varphi+\psi)-P(\varphi).
\]
 Note that $Q _\varphi$ is a continuous and convex functional.
Within the framework of the theory of conjugating functions (see for
example~\cite{AubEke:84}), $Q_\varphi$ can be characterized by
\[
Q_\varphi(\psi) =
\sup_{\nu\in\cP} \left(\int\psi \,d\nu - I_\varphi(\nu)\right),
\]
 where $I_\varphi$ is the convex conjugate  of $Q_\varphi$  defined by
\begin{equation}\label{Idef}
    I_\varphi(\mu) \eqdef
    \sup_{\psi\in C^0(T^1M,\bR)}\left( \int\psi \,d\mu - Q_\varphi(\psi)
        \right)
\end{equation}
for all $\mu\in\cP$ and $I_\varphi(\mu) =\infty$ for any other signed measure $\mu$. 
Since $I_\varphi\colon \cP \to \bR$ is a pointwise supremum of
continuous and affine functions, it is a  lower semi-continuous  and convex
functional.
Given $\nu\in \cP$, we call
\begin{equation}\label{defunas}
	\widehat h(\nu) \eqdef
	\inf_{\psi\in C^0(T^1M,\bR)}
	\left(P(\psi)-\int\psi \,d\nu \right)
\end{equation}
the \emph{generalized entropy} of $f$ with respect to $\nu$. 
It follows from the  definition that $h(\nu)\le\widehat h(\nu)$ for every $\nu\in \cM$. 
 A (not necessarily invariant) measure  $\mu\in\cP$ is called a \emph{generalized equilibrium state} for $\varphi$ if
 $P(\varphi)=\widehat h(\mu)+\int\varphi\,d\mu$. 
This terminology is justified by the  {\em dual variational
principle} $ h(\nu) = \widehat h(\nu)$  \cite[Chapter 9.4]{Wal:81}. 
Observe that
\begin{eqnarray*} 
P(\varphi)  -
\widehat h(\mu)-\int\varphi \,d\mu  
&=& P(\varphi) - \int\varphi \,d\mu
+ \sup_{\psi}\left(\int(\psi+\varphi) \,d\mu - P(\psi+\varphi)\right)\\
&= & \sup_{\psi}\left(\int \psi \,d\mu
- P(\psi+\varphi)+P(\varphi)\right) = I_\varphi(\mu)
\end{eqnarray*}
Thus,  for  all $\mu\in\cM$, the equality $h(\mu)=\widehat{h}(\mu)$ implies \begin{equation}
 I_\varphi(\mu) =
P(\varphi) - \left(h(\mu)+\int\varphi \,d\mu\right)\ge 0.
 \end{equation}     
 Moreover, for $\mu\in\cM$ we have
$I_\varphi(\mu)=0$ if, and only if, $\mu $ is an equilibrium state for $\varphi$. 
Therefore, on can think of the functional $I_\varphi$   as a 
``distance" from $\mu$ to the set of all generalized equilibrium states of $\varphi$.

Following closely~\cite[Section 2]{Pol:95}, we obtain the following result.
\begin{lemma}\label{l:17} Let $M$ be a smooth compact rank one surface. 
	Let $\varphi\colon T^1M\to\bR$ be a continuous potential, and $\cK\subset\cM$ be a compact set. Then we have
	\[
		\limsup_{T\to\infty}
		\frac 1 T\log 
		\sum_{\beta\in\Pi(T-1,T),\nu_\beta\in\cK}
		\hspace{-0.5cm}e^{\int\varphi\,d\nu_\beta}
		\le P_{top}(\varphi)-\inf_{\nu\in\cK}I_{\varphi}(\nu),
	\]	
	where $\nu_\beta$ denotes the invariant probability measure supported on the periodic orbit which projects to $\beta$.
\end{lemma}

\begin{proof}
  Let $\rho\eqdef\inf_{\nu\in\cK}I_\varphi(\nu)$. For all $\nu\in\cK$, by~\eqref{Idef} we have
  \begin{equation}
  \rho\le\sup_{\psi}\left(\int\psi \,d\nu-Q_\varphi(\psi)\right).
  \end{equation}
  Hence, given $\nu\in \cK$ and $\varepsilon>0$ there exists $\psi=\psi(\nu,\varepsilon)\in C^0(T^1M,\bR)$
such that $\rho -\varepsilon <  \int \psi \,d\nu -Q_\varphi(\psi)$.
  Thus, we obtain that
  \begin{equation}
  \cK \subset \bigcup_{\psi}
  \left\{\nu\in\cM\colon \int\psi \,d\nu >
    Q_{\varphi}(\psi)+\rho-\varepsilon\right\}.
  \end{equation}
It is a covering of $\cK$ by open sets. 
  By compactness there exists a finite cover
  $\cU_1$, $\ldots$, $\cU_N$ of $\cK$ determined by functions $\psi_1$,
  $\ldots$, $\psi_N$ through
  \begin{equation}
  \cU_i\eqdef   \left\{\nu\in\cM\colon
    \int\psi_i \,d\nu - Q_\varphi(\psi_i)-\rho +\varepsilon
    >0 \right\}.
  \end{equation}
  We get
  \[
\begin{split}
&  \sum\limits_{\beta\in\Pi(T-1,T),\,\,\nu_\beta\in \cK} e^{\int\varphi\,d\nu_\beta}
	\le
    \sum_{i=1}^N
    \sum\limits_{\beta\in\Pi(T-1,T),\,\,\nu_\beta\in \cU_i}e^{\int\varphi\,d\nu_\beta}\\
  & \phantom{\sum}<
    \sum_{i=1}^N
    \sum\limits_{\beta\in\Pi(T-1,T),\,\,\nu_\beta\in \cU_i}e^{\int\varphi\,d\nu_\beta}\cdot
    \exp \left[T\left(\int\psi_i \,d\nu_\beta- Q_\varphi(\psi_i)-\rho
    			+\varepsilon\right)\right]\\
    & \phantom{\sum} \le
    \sum_{i=1}^N \exp [-T(Q_\varphi(\psi_i)+\rho-\varepsilon)]
    \sum\limits_{\beta\in\Pi(T-1,T) }
    e^{\int(\varphi+\psi_i)\,d\nu_\beta}.
  \end{split}\]
 Now~\eqref{inequality-pressure-manifolds} in Theorem~\ref{equality-pressure-surfaces} applied to the pressure of $\varphi+\psi_i$ gives us
  \begin{equation}
\begin{split}
    \limsup_{T\to\infty}&\frac{1}{T}\log
    \sum\limits_{\beta\in\Pi(T-1,T),\nu_\beta\in \cK}e^{\int\varphi\,d\nu_\beta}\\
    &\le \max_{1\le i\le N}
    \left\{ -Q_\varphi(\psi_i)-\rho+\varepsilon+P_{top}(\varphi+\psi_i)\right\}\\
    &= P_{top}(\varphi)-\rho+\varepsilon                .
  \end{split}\end{equation}
  Since $\varepsilon>0$ was arbitrary, this concludes the proof. 
\end{proof} 

We now are ready to prove Theorem~\ref{large-deviations}.

\begin{proof}[Proof of Theorem~\ref{large-deviations}]
We first prove item 1.
Consider the compact subset 
\[
	\cK\eqdef
		\{\mu\in\cM\colon \chi(\mu)\le\alpha(\varphi)-\delta\}.
\]	 
For  $\nu\in\cK$, using definitions of $\cK$, $I_\varphi$ and Ruelle's inequality, we get
\[
		I_\varphi(\nu)
		 \ge P_{top}(\varphi)-\max\varphi -h(\nu)\\
		 =\alpha(\varphi)-h(\nu)
		\ge
		\alpha(\varphi)-\chi(\nu) \ge\delta.
\] 
It follows that  $\inf_{\nu\in\cK}I_\varphi(\nu)\ge\delta$. Now Lemma~\ref{l:17} implies 
\begin{equation}\label{otherhand}
		\limsup_{T\to\infty}\frac 1 T\log
		\sum_{\beta\in\Pi(T-1,T),\chi(\beta)\le\alpha-\delta}e^{\int\varphi\,d\nu_\beta}
		\le P_{top}(\varphi)-\delta.
\end{equation}
Theorem~\ref{equality-pressure-surfaces}~\eqref{eqtheoremsurface} gives $P_{top}(\varphi)=P_{Gur}(\varphi)$.
Hence there is a subsequence $T_k\to\infty$ such that
\begin{equation}\label{together}	
		\lim_{k\to\infty}\frac{1}{T_k}\log
		\sum_{\beta\in\Pi(T_k-1,T_k)}e^{\int\varphi\,d\nu_\beta}
		=P_{top}(\varphi).
\end{equation}		
	So for any $\varepsilon>0$ there exists $T_0\ge 1$ so that for every $k\ge 1$ with $T_k\ge T_0$ we have
	\[
		\sum_{\beta\in\Pi(T_k-1,T_k)}e^{\int\varphi\,d\nu_\beta}
		\ge e^{T_k(P_{top}(\varphi)-\varepsilon)}.
	\]
	On the other hand, if $T_0$ is large enough, then by~\eqref{otherhand} we also have 
	\[
		\sum_{\beta\in\Pi(T_k-1,T_k),\chi(\beta)\le\alpha-\delta}
			e^{\int\varphi\,d\nu_\beta}
		\le e^{T_k(P_{top}(\varphi)-\delta+\varepsilon)}.
	\]
	Combining the  two above inequalities, we obtain 
	\[
		\sum_{\beta\in\Pi(T_k-1,T_k),\chi(\beta)>\alpha-\delta}
			e^{\int\varphi\,d\nu_\beta}
		\ge
		e^{T_k(P_{top}(\varphi)-\varepsilon)}\left(
			1-e^{T_k(-\delta+2\varepsilon)}\right).
	\]
	For $\varepsilon<\delta/2$, this implies
	\[
		\limsup_{T_k\to\infty}\frac {1}{T_k}\log
		\sum_{\beta\in\Pi(T_k-1,T_k),\chi(\beta)>\alpha-\delta}
			e^{\int\varphi\,d\nu_\beta}
		\ge P_{top}(\varphi)-\varepsilon\,.
	\] 
The left hand side is clearly smaller than the Gurevic pressure $P_{Gur}(\varphi)$, which, by 
 Theorem~\ref{equality-pressure-surfaces}, coincides with $P_{top}(\varphi)=P_\cM(\varphi)$. 
	As $\varepsilon$ was arbitrary, together with~\eqref{together} this completes the proof of item 1.
	
Now we prove item 2. We refer to section \ref{periodic-orbits} for geometric observations. 
Fix some $\eta>0$. 
By uniform continuity of $\varphi$ and of its lift $\widetilde{\varphi}$ to $T^1\widetilde{M}$,
we can choose $r\in(0,\min\{1/4,\eta,\rho\})$, such that for any two vectors 
$v,w$ in $T^1\widetilde{M}$, $d(v,w)\le r$ implies 
$|\widetilde{\varphi}(v)-\widetilde{\varphi}(w)|\le \eta$.
 Let $\{B(x_i,r)\}_{i=1}^N$ be a finite cover of $M$.
By   item 1 above, we have 
\begin{equation}\label{eq:proofTT1}
	\limsup_{T\to +\infty}\frac{1}{T}\log \sum_{\beta\in\Pi(T-1,T),\chi(\beta)>\alpha(\varphi)-\delta} e^{\int\varphi\,d\nu_\beta}
	=P_{top}(\varphi)
\end{equation}

Let $T\ge T_0+1 +2r$, and $\beta\in\Pi(T-1, T)$ a periodic orbit of the geodesic flow 
with $\chi(\beta)>\alpha-\delta$. The closed geodesic associated to $\beta$ intersects some $B(x_i,r)$. 
Let $\gamma\in \Gamma$ be an isometry whose axis projects to $M$ 
onto this closed geodesic, and whose translation length is $\ell(\beta)$. 
One can lift $x_i$ and the closed geodesic associated to $\beta$
in  such a way that this lift intersects $B(\widetilde{x}_i,r)$ and $B(\gamma\widetilde{x}_i,r)$, where $\widetilde{x}_i$ is the lift of $x_i$. 
Therefore, the geodesic from $\widetilde{x}_i$ to $\gamma\widetilde{x}_i$ projects on $M$
 to a loop $(\beta_\gamma(t))_{0\le t\le T_\gamma}$ with $\beta_\gamma(0)=\beta_\gamma(T_\gamma)=x_i$. 
Moreover, a simple triangular inequality gives $\lvert T_\gamma -\ell(\beta)\rvert\le 2r$. 
By construction of $\beta_\gamma$, by uniform continuity of $\varphi$, 
and elementary considerations in nonpositive curvature (see Section \ref{periodic-orbits}), 
as $\ell(\beta)\in (T-1, T)$, 
we get
$$
	\left\lvert\int_\beta\varphi \,d\nu_\beta
	-\int_0^{T_\gamma}\varphi( \beta_\gamma'(t))\,dt\right\rvert
	\le \eta\,\ell(\beta)+2r\,\|\varphi\|_\infty\le \eta T+2\eta\, \lVert\varphi\rVert_\infty
$$
A different closed orbit $\beta$ may lead to  a different point $\widetilde{x}_i$ from the cover. Summing over $i$, using the fact that
$\int_0^{T_\gamma}\varphi( \beta_\gamma'(t))dt= \int_{\widetilde{x}_i}^{\gamma \widetilde{x}_i}\widetilde{\varphi}$, we obtain 
$$
\sum_{\beta\in\Pi(T-1,T),\,\,\chi(\beta)>\alpha(\varphi)-\delta} e^{\int\varphi\,d\nu_\beta}
\le e^{\eta T + 2\eta \|\varphi\|_\infty}
\sum_{i=1}^{N}\sum_{\gamma\in\Gamma_{\alpha(\varphi)-\delta},\,\,d(\widetilde{x}_i,\gamma \widetilde{x}_i)\le T} 
e^{\int_{\widetilde{x}_i}^{\gamma \widetilde{x}_i}\widetilde{\varphi}}\,.
$$
Now taking $\limsup_{T\to\infty}\frac1T\log$, with~\eqref{eq:proofTT1} we obtain 
\[	P_{top}(\varphi) 
	\le \max_{1\le i\le N} \limsup_{T\to\infty}
\frac 1T\log \sum_{\gamma\in\Gamma_{\alpha(\varphi)-\delta},d(\widetilde{x}_i,\gamma \widetilde{x}_i)\le T} e^{\int_{\widetilde{x}_i}^{\gamma \widetilde{x}_i}\widetilde \varphi}
	 + \eta \,.
\]
Thanks to property~\eqref{inequality-pressure-manifolds} in Theorem~\ref{equality-pressure-surfaces}, the latter term is bounded from above by $\max_{i=1,\ldots,N}\delta_{\Gamma,\varphi,x_i}+\eta\le \delta_{\Gamma,\varphi}+\eta $.
As $\eta$ can be taken arbitrarily small, this finishes the proof. 
\end{proof}
 
\bibliographystyle{amsplain}

 \end{document}